\documentclass[preprint,11pt]{elsarticle}
\usepackage{amssymb}
\usepackage{amsthm}
\usepackage{graphics}
\usepackage{epsfig}
\usepackage{amssymb}
\usepackage{graphicx}
\usepackage{subfigure}
\usepackage{color}
\usepackage{comment}
\usepackage{tikz}
\usepackage{tikz,pgfplots}
\pgfplotsset{compat=1.12,axis lines=center}
\usepgfplotslibrary{fillbetween}
\usetikzlibrary{patterns}
\usepackage{hyperref}
\usepackage{a4wide}
\usepackage{amsmath}
\usepackage{amsfonts}
\usepackage{enumerate}
\usepackage{multicol}

\newcommand{\cyc}{\operatorname{Cyc}}	
\newcommand{\aut}{\operatorname{Aut}}	

\newcommand{\PSL}{\operatorname{PSL}}
\newcommand{\PGL}{\operatorname{PGL}}
\newcommand{\PSU}{\operatorname{PSU}}
\newcommand{\SL}{\operatorname{SL}}
\newcommand{\GL}{\operatorname{GL}}
\newcommand{\Sz}{\operatorname{Sz}}
\newcommand{\soc}{\operatorname{soc}}
\newcommand{\CG}[2]{C_{#1}(#2)}

\makeatletter
\def\ps@pprintTitle{%
	\let\@oddhead\@empty
	\let\@evenhead\@empty
	\let\@oddfoot\@empty
	\let\@evenfoot\@oddfoot
}
\makeatother

\hypersetup{
	colorlinks=true,
	linkcolor=blue,
	filecolor=dviolet,      
	urlcolor=blue,
} 

\newtheorem*{theoremaux}{Theorem \theoremauxnum}
\gdef\theoremauxnum{1}

		\newtheorem{theorem}{\bf Theorem}[section]
	\newtheorem{lemma}[theorem]{\bf Lemma}
	
	\newtheorem{proposition}[theorem]{\bf Proposition}
	\newtheorem{corollary}[theorem]{\bf Corollary}

\journal{~}

\begin{document}

	\begin{frontmatter}
		
		
		
		\title{On the finite group whose proper enhanced power graph is claw-free}

		
		\author{Sudip Bera}
		\ead{sudip\_bera@dau.ac.in} 
		\address{Faculty of Mathematics,\\ 
		Dhirubhai Ambani University, Gandhinagar \\ India} 
		\author{Andrea Lucchini}
		\ead{lucchini@math.unipd.it}
		%
		\address{Università di Padova, Dipartimento di Matematica “Tullio Levi-Civita”, Via Trieste 63, 35121 Padova,\\ Italy.}
		
	
		\begin{abstract}
\noindent Let $G$ be a finite group. The \emph{enhanced power graph} of $G$,
denoted by $\mathcal{E}(G)$, is the graph with vertex set $G$ in which two 
vertices $u$ and $v$ are adjacent if and only if there exists an element 
$w \in G$ such that both $u$ and $v$ belong to $\langle w \rangle$. The 
\emph{proper enhanced power graph} of $G$, denoted by $\mathcal{E}^{**}(G)$, 
is the subgraph of $\mathcal{E}(G)$ induced by the non-dominating vertices.

The main objective of this paper is to investigate finite groups whose proper 
enhanced power graph is claw-free, that is, contains no induced subgraph 
isomorphic to the complete bipartite graph $K_{1,3}$. We first prove that 
$\mathcal{E}(G)$ is claw-free if and only if $G$ is cyclic. The set of 
dominating vertices of $\mathcal{E}(G)$ forms a cyclic subgroup of the center 
of $G$, namely the \emph{cyclicizer} $\cyc(G)$ of $G$. This allows us to give 
a precise description of the structure of $G/\cyc(G)$ when $\mathcal{E}^{**}(G)$ 
is claw-free. If $G$ is solvable but not nilpotent, then $G$ is metacyclic, or $G/\cyc(G)$ is either 
a Frobenius group or a $2$-Frobenius group. If $G$ is non-solvable, then 
$G/\cyc(G)$ is isomorphic to $\PSL(2,q)$ or $\PGL(2,q),$ and this allows us to give a complete classification of the non-solvable
groups whose proper enhanced power graph is claw-free.
\end{abstract}

		
	\end{frontmatter}

    \section{Introduction}
A graph is called a \emph{claw} if it is isomorphic to the complete bipartite graph $K_{1,3}$. A graph is said to be \emph{claw-free} if it contains no induced subgraph isomorphic to a claw.
Claw-free graphs constitute one of the most important and extensively studied classes of graphs in graph theory. This family encompasses several well-known graph classes, such as line graphs, complements of triangle-free graphs, and proper interval graphs \cite{clawfreesurvey}. It also includes the graphs of several polyhedra and polytopes: the graph of the tetrahedron and, more generally, of any simplex; the graph of the octahedron and, more generally, of any cross-polytope; the graph of the regular icosahedron; and the graph of the 16-cell \cite{clawfree2005}.
 The systematic study of claw-free graphs originated from Beineke's celebrated characterization of line graphs, which highlighted the structural significance of forbidding the claw as an induced subgraph \cite{Beineke1970}. Subsequently, during the late 1970s and early 1980s, claw-free graphs attracted considerable attention due to their rich structural and algorithmic properties \cite{Edmonds1965,Minty1980}.

Algebraic graph theory provides a powerful framework for studying algebraic 
structures through graph-theoretic techniques. By associating graphs with 
algebraic objects such as groups, semigroups, rings, and other algebraic 
systems, researchers have gained valuable insights into both the structural 
properties of graphs and the underlying algebraic objects. Such graph 
representations facilitate the investigation of graph invariants, the 
characterization of algebraic structures via graph isomorphisms, and the 
discovery of deep connections between graph theory and abstract algebra. 
Among the graphs associated with groups is the \emph{commuting graph}, 
originally introduced by Erd\H{o}s, whose vertices are the elements of a 
group, with two distinct vertices adjacent if and only if the corresponding 
elements commute \cite{braurflower,MorganParker2013}. Several graphs 
associated with a finite group encode its generation properties. One such 
example is the \emph{generating graph}, in which two vertices $x$ and $y$ 
are adjacent if and only if $G=\langle x,y\rangle$~\cite{LucchiniMaroti2009}. 
Observe that the generating graph is empty whenever $G$ is not $2$-generated. 
To extend the study of generation to arbitrary finite groups, the 
\emph{independence graph} was introduced. Its vertices are the elements of 
$G$, and two vertices $x$ and $y$ are adjacent precisely when $\{x,y\}$ is 
contained in a minimal generating set of $G$~\cite{Lucchini2020}. Other 
graphs encoding the group's internal structure include the \emph{enhanced 
	power graph} \cite{firstenhcedpwrstrctreaacns1, enhancedpwrgrapbb3} and the 
\emph{power graph} \cite{Bera-Line-graph-comm-alg, jgt-cameron, 
	pwrgraphoffntgrpgc1}. Readers interested in a broader overview of this topic 
are referred to the survey paper by Cameron~\cite{PJC-Graphs-def-grp}.

The power graph of a group was first introduced by Kelarev and Quinn in 2002 
as a directed graph, where two elements are related whenever one is a power 
of the other. Later, in 2009, Chakrabarty et al.\ defined the undirected 
power graph in \cite{undpwrgraphofsemgmainsgc1}. The undirected power graph 
(or simply power graph) of a group $G$, denoted by $\mathcal{P}(G)$, is the 
simple undirected graph whose vertex set is $G$, and in which two distinct 
vertices $u$ and $v$ are adjacent if and only if one is a positive power of 
the other. The reduced power graph of $G$, denoted by $\mathcal{P}^{*}(G)$, 
is the induced subgraph of $\mathcal{P}(G)$ obtained by deleting the identity 
element of $G$. Since its introduction, the power graph has become an active 
area of research owing to its close relationship with both group-theoretic 
and graph-theoretic properties. In 2017, the authors of 
\cite{firstenhcedpwrstrctreaacns1} introduced the enhanced power graph, which 
lies between the power graph and the commuting graph, as a means of measuring 
how close the former is to the latter.
This graph preserves more structural information about the underlying group 
while remaining mathematically tractable, making it an important object of 
study in contemporary algebraic graph theory. Much work has been done 
recently on various properties of the enhanced power graph of finite groups. 
In \cite{firstenhcedpwrstrctreaacns1}, the finite groups for which any 
arbitrary pair of these three graphs~--- power, commuting, and enhanced~--- 
coincide are characterized. In \cite{Zahirovienhnacedpwrgraph}, it was proved 
that finite groups with isomorphic enhanced power graphs also have isomorphic 
directed power graphs. Bera et al.\ in \cite{enhancedpwrgrapbb3} studied the 
completeness, dominatability, and several other properties of the enhanced 
power graph. Ma and She in \cite{ma-she} derived the metric dimension, while 
Hamzeh et al.\ in \cite{Hamzeh-ashrafi} determined the automorphism groups 
of enhanced power graphs of finite groups.

Motivated by the significance of both claw-free graphs and enhanced power graphs, the present work investigates the claw-free property of enhanced power graphs. It explores the structural conditions on groups under which their enhanced power graphs are claw-free. In \cite{claw-free-pwr-graph}, the authors classified all the groups whose reduced power graph is claw-free. 
 In this paper, our topic of interest is the claw-free structure of the proper enhanced power graph. 

\section{Main results}
\label{sec:main_results}
The \emph{enhanced power graph} of $G,$
denoted by $\mathcal{E}(G),$ is the graph with vertex set $G,$ in which two vertices $u$ and $v$ are joined if and only
if there exists an element $w \in G$ such that both $u \in \langle w \rangle $ and $v \in  \langle w \rangle,$ or, equivalently, if and only if $\langle u, v\rangle$ is cyclic.
The \emph{proper enhanced power graph} of $G,$ denoted by $\mathcal{E}^{**}(G),$ is the graph obtained by deleting all the dominating vertices from the enhanced power graph $\mathcal{E}(G).$     
The main objective of this paper is to investigate finite groups whose proper enhanced power graph is claw-free, but we first classify all finite groups for which the enhanced power graphs are claw-free. 
\begin{theorem}\label{Thm:enh-claw-free-imply-uniuq-subgroup}
Let \(G\) be a finite group. Then the enhanced power graph \(\mathcal{E}(G)\) is claw-free if and only if $G$ is cyclic.
\end{theorem}

The presence of possible dominating vertices makes it harder to classify groups whose graph $\mathcal{E}^{**}(G)$ is claw-free. A helpful observation, however, is that — as described in Section \ref{preliminar} — the set of dominant vertices forms a cyclic subgroup of the center, known as the cyclicizer of $G$ and denoted by $\cyc(G)$. Taking into account the properties of this subgroup, we can provide the following characterization of finite solvable groups whose proper enhanced power graph is claw-free.

\begin{theorem}\label{solvable}
	Let $G$ be a finite solvable group such that $\mathcal{E}^{**}(G)$ 
	is claw-free. Let $F(G)$ be the Fitting subgroup of $G$ and 
let $F_2(G)$ denote the preimage 
	in $G$ of $F(G/F(G))$.
	Then either $F(G)$ and $G/F(G)$ are both cyclic or the  following hold:
	\begin{enumerate}
		\item $F(G)/\cyc(G)$ is a non-cyclic $p$-group.
		\item If $G$ is not nilpotent, then $F_2(G)/\cyc(G)$ is a Frobenius group with 
		Frobenius kernel $F(G)/\cyc(G)$, and $F_2(G)/F(G)$ is cyclic.
		\item If $G \neq F_2(G)$, then $G/F_2(G)$ is a $p$-group and 
		$G/\cyc(G)$ is a $2$-Frobenius group.
	\end{enumerate}
\end{theorem}

As a corollary, we can obtain the following characterization of nilpotent groups whose enhanced power graph is claw-free.

\begin{corollary}\label{nilpotent}
Let $G$ be a finite non-cyclic nilpotent group. Then the proper enhanced power graph $\mathcal{E}^{**}(G)$ is claw-free if and only if the following conditions hold:
\begin{enumerate}[(i)]
\item $G \cong P \times C_n$, where $P$ is a non-cyclic $p$-group and $\gcd(|P|,n)=1$;
\item $\mathcal{E}^{**}(P)$ is claw-free.
\end{enumerate}
\end{corollary}

The previous corollary essentially reduces the study of nilpotent groups $G$ with $\mathcal{E}^{**}(G)$ claw-free to the case where $G$ is a non-cyclic $p$-group. In this setting, the following result holds.

\begin{theorem}\label{pgroups}
Let $G$ be a non-cyclic finite 
$p$-group. Then $\mathcal{E}^{**}(G)$ is claw-free if and only if $G$ is a $p$-group of one of the following types:
\begin{enumerate}[(i)]
\item a non-cyclic $p$-group of exponent $p$;
\item $C_{2}\times C_{4}$;
\item $C_{4}\times C_{4}$;
\item a dihedral $2$-group;
\item a non-abelian $2$-group of exponent $4$ containing no proper subgroup isomorphic to $Q_8$;
\item a generalized quaternion group of order $2^k$, where $k \geq 3$.
\end{enumerate}
\end{theorem}

In the non-solvable case, we are able to give a complete and definitive result. First, in Theorem~\ref{thm:almost-simple} we prove that, after factoring out the subgroup of dominating vertices, the resulting quotient group is almost simple. Then, in Theorem~\ref{PSL}, we prove that the only almost simple groups that can arise are $\PSL(2,q)$ and $\PGL(2,q)$ with $q\geq 5.$ This allows us to establish the following result:

\begin{theorem}\label{nonsolvable}
	\label{thm:central-extensions-PSL-PGL}
	Let $G$ be a non-solvable finite group. Then $\mathcal E^{**}(G)$ claw-free if and only if $G\cong X\times Y,$ 
	where $X$ is cyclic, 
	$\gcd(|X|,|Y|)=1,$ and $Y$ is one of the following:
	\begin{enumerate}
		\item 
	$
		Y\cong\operatorname{PSL}(2,q)$ or
		$		Y\cong\operatorname{SL}(2,q), 
		$ with $q\geq 5.$
		
		\item
		$
		Y\cong\operatorname{PGL}_2(q)
		$
		or
		$
		Y\cong
		\operatorname{SL}^{(2)}(2,q),
		$
		where $q\geq 5$ is odd and
		\[
		1\longrightarrow C_2
		\longrightarrow
		\operatorname{SL}^{(2)}(2,q)
		\longrightarrow
		\operatorname{PGL}(2,q)
		\longrightarrow1
		\]
		is the unique stem double cover of
		$\operatorname{PGL}(2,q)$ whose Sylow $2$-subgroups are
		generalized quaternion.
	\end{enumerate}
\end{theorem}

\section{Some preliminary results}\label{preliminar}
\begin{proof}[Proof of Theorem \ref{Thm:enh-claw-free-imply-uniuq-subgroup}]
Let $G$ be a minimal counterexample. Since $\mathcal{E}(H)$ is claw-free for every 
$H\leq G$, all the proper subgroups of $G$ are cyclic. Hence $G$ is minimal non-cyclic. 
Groups with this property have been classified by Miller and Moreno in \cite{miller_moreno}. There are three 
possibilities: $G\cong C_p\times C_p$, $G\cong Q_8$, $G\cong C_p\rtimes C_{q^m}$. 
In all these cases $\mathcal{E}(G)$ contains a claw $\{y, x_1, x_2, x_3\}$. In the 
first case, take $y=1$ and $x_1, x_2, x_3$ generators of three different subgroups of 
order $p$; in the second case, take $x_1, x_2, x_3$ generators of the three cyclic 
subgroups of order $4$ and $y=x_1^2=x_2^2=x_3^2$; in the third case, take $y=1$ and 
$x_1, x_2, x_3$ generators of three different Sylow $q$-subgroups.
\end{proof}

For a finite group $G$, the set $$\cyc(G)=\{x\in G\mid \langle x, y\rangle \text{ is cyclic for every $y\in G$}\}$$ of dominating vertices of $\mathcal E(G)$ has been studied by Patrick and
Wepsic, who called it the cyclicizer of $G$. The following holds:

\begin{lemma}[{\cite[Theorem 9.1 (b)]{PJC-Graphs-def-grp}}]\label{Lemma:K-cyclic-enh}
 Let $G$ be a finite group. Then $\cyc(G)$ is a cyclic subgroup of $Z(G)$. More precisely, $\cyc(G)$ is the product of the Sylow $p$-subgroups of $Z(G)$ for $p \in \tilde \pi,$ where $\tilde \pi$ is the set of
 primes $p$ for which the Sylow $p$-subgroup of $G$ is cyclic or generalized quaternion.
\end{lemma}

\begin{lemma}\label{Lemma:Cyclic-iff-cyclic}
	Let $x, y \in G.$ Then $\langle x, y\rangle$ is cyclic if and only if 
	$\langle x, y\rangle\cyc(G)/\cyc(G)$ is cyclic.                                         
\end{lemma}

\begin{proof}
	Let $C = \cyc(G)$.
	By Lemma~\ref{Lemma:K-cyclic-enh}, $C = \langle c\rangle$ is cyclic.       
	Since $C$ contains dominating vertices in $\mathcal{E}(G)$, the subgroup
	$\langle c, g\rangle$ is cyclic for every $g \in G$.
	If $\langle x, y\rangle = \langle z\rangle$, then clearly
	$\langle x, y\rangle C = \langle z, c\rangle$ is cyclic.
	Now let $\langle x, y\rangle C = \langle g\rangle C$. Then
	$\langle x, y\rangle \leq \langle g\rangle C = \langle g, c\rangle$, 
	which is cyclic. Hence $\langle x, y\rangle$ is cyclic.
	This completes the proof.
\end{proof}

\begin{corollary}\label{behacyc}
	Let $G$ be any finite group. Then the following hold:
	\begin{enumerate}
		\item For every $x, y \in G$, $x$ and $y$ are adjacent in 
		$\mathcal{E}^{**}(G)$ if and only if $x\cyc(G)$ and $y\cyc(G)$ 
		are adjacent in $\mathcal{E}^{**}(G/\cyc(G))$.
		\item The identity element is the only dominating element in 
		$\mathcal{E}^{**}(G/\cyc(G))$.
		\item $\mathcal{E}^{**}(G)$ is claw-free if and only if 
		$\mathcal{E}^{**}(G/\cyc(G))$ is claw-free.
	\end{enumerate}
\end{corollary}

\begin{proof}
	Let $C = \cyc(G)$. Part~(1) and Part~(3) follow immediately from 
	Lemma~\ref{Lemma:Cyclic-iff-cyclic}. For Part~(2), note that 
	$xC$ is a dominating vertex in $\mathcal{E}^{**}(G/C)$ if and only 
	if $\langle xC, yC\rangle$ is cyclic for every $y \in G$. By 
	Lemma~\ref{Lemma:Cyclic-iff-cyclic}, this holds if and only if 
	$\langle x, y\rangle$ is cyclic for every $y \in G$, that is, 
	$x \in C$, so $xC = C$ is the identity of $G/C$.
\end{proof}

\section{Proofs of Theorem \ref{solvable}, Corollary \ref{nilpotent} and Theorem \ref{pgroups}}

\begin{theorem}\label{Thm:F(G)-cyclic}
   Let $G$ be a finite solvable group whose Fitting subgroup $F=F(G)$ is
   cyclic.  If $\mathcal E^{**}(G)$ is claw-free, then $G/F$ is cyclic.
\end{theorem}

\begin{proof}
Since $G$ is solvable, the Fitting subgroup contains its own
centralizer. Because $F$ is abelian, this gives
$
   \CG{G}{F}=F.
$
Consequently conjugation induces a faithful embedding
$
   A:=G/F\hookrightarrow \text{Aut}(F).
$
In particular, since $F$ is cyclic, $A$ is abelian.  It therefore suffices to prove that every
Sylow subgroup of $A$ is cyclic.

Suppose, to the contrary, that a Sylow $q$-subgroup of $A$ is noncyclic for
some prime $q$.  Since $A$ is abelian, it contains a subgroup
$
   E\cong C_q\times C_q.
$
Write $F=F_q\times F_{q'}$, where $F_q$ is the Sylow $q$-subgroup of $F$.
We first claim that $E$ acts faithfully on $F_{q'}$.

Indeed, let
$ B=\CG{E}{F_{q'}}$
and suppose that $B\ne 1$.  Let $R$ be the full inverse image of $B$ in $G$.
Since $A$ is abelian, $R$ is normal in $G$.  Moreover, $F_{q'}\le Z(R)$ and
$R/F_{q'}$ is a $q$-group.  It follows that
$R=F_{q'}\times Q$
for a (necessarily unique) Sylow $q$-subgroup $Q$ of $R$.  Hence
$Q=O_{q}(R)$ is characteristic in $R$ and therefore normal in $G$. 
Hence $Q\leq F,$ in contradiction with $QF/F\cong B\ne 1$.

For every prime $p\ne q$, the $q$-subgroups of
$\text{Aut}(F_p)$ are cyclic.  Hence
every nontrivial homomorphism
$  E\longrightarrow\text{Aut}(F_p)
$
has a subgroup of order $q$ as its kernel.  Since $\CG{E}{F_{q'}}=1,$
the intersection of these
kernels, as $p$ ranges over the prime divisors of $|F_{q'}|$, is trivial.
There must therefore be two primes $p,r\ne q$ such that the corresponding
kernels $L_p$ and $L_r$ are distinct.
Choose
$1\ne \bar{x}\in L_p\setminus L_r.
$
Thus $\bar{x}$ centralizes $F_p$ and acts nontrivially on $F_r$. Choose also
$\bar{y}\in E\setminus L_p$, so that $\bar{y}$ acts nontrivially on $F_p$.
Let $z$ be the element of order $p$ in the cyclic group $F_p$. Then
$\bar{x}$ centralizes $z$, whereas $\bar{y}$ does not.  In particular,
$z\notin\cyc(G).
$

Choose a $q$-element $x\in G$ such that $xF=\bar x.$  Then $x$ centralizes
$z$ but acts nontrivially on $F_r$.  Since this is a
coprime automorphism of the cyclic $r$-group $F_r$ of order $q$, it is
fixed-point-free on $F_r$.  Choose $1\ne f\in F_r$.  Then
$
   x^f\ne x.
$
Furthermore, $x$ and $x^f$ do not commute.  Otherwise $x$ would centralize
the nonidentity element $x^{-1}x^f\in F_r$, contradicting
$C_{F_r}(x)=1$.  Thus
$\langle x,x^f\rangle$ is not cyclic

Let $a$ be a generator of $F$.  Since $z,a\in F$, the subgroup
$\langle z,a\rangle$ is cyclic.  Also, $z$ commutes with both $x$ and $x^f$;
as $|z|=p$ and $x,x^f$ are $q$-elements with $p\ne q$, both
$
   \langle z,x\rangle$
and $\langle z,x^f\rangle
$
are cyclic.
In contrast, neither $\langle a,x\rangle$ nor $\langle a,x^f\rangle$ is
cyclic.  Indeed, cyclicity would imply that $x$, respectively $x^f$,
centralizes the generator $a$ of $F$, contrary to
$C_G(F)=F$. We have proved that $\{z,a,x,x^f\}$
induces a claw with center $z$.
All four vertices belong to $G\setminus\cyc(G)$.
Indeed  $a$ is not adjacent to $x$; and $x$ and $x^f$ are not
adjacent to $a$.  We have therefore found an induced claw in
$\mathcal E^{**}(G)$, a contradiction.

It follows that $A$ contains no subgroup isomorphic to $C_q\times C_q$ for
any prime $q$.  Since $A$ is abelian, all its Sylow subgroups are cyclic.
Therefore $A=G/F$ is cyclic.
\end{proof}

\begin{lemma}\label{excludedsbgs}
	Suppose that $\cyc(G) =1$. If $\mathcal{E}^{**}(G)$ is claw-free, 
	then $G$ contains no subgroup isomorphic to a generalized quaternion 
	group or to $C_p \times C_p \times C_q$.
\end{lemma}

\begin{proof}
	Suppose $G$ has a subgroup isomorphic to $C_p\times C_p\times C_q.$ Then $G$ contains three distinct cyclic subgroups $\langle a_1\rangle, \langle a_2\rangle, \langle a_3\rangle$ of order $p,$ and a cyclic subgroup $\langle b\rangle$ of order $q.$ Since $\cyc(G)=1,$ the elements $a_1, a_2, a_3, b$ are not dominating vertices and induce a claw with center $b$ and leaves  $a_1, a_2, a_3,$ a contradiction. Similarly,  if $G$ has a subgroup isomorphic to a generalized quaternion group, then $\mathcal E^{**}(G)$ contains a claw $a_1,a_2,a_3,b$ in which the leaves $a_1,a_2,a_3$ have order 4 and $b=a_1^2=a_2^2=a_3^2.$
\end{proof}

\begin{lemma}\label{frobcent}
	Suppose that a finite group $G$ contains an element $x$ and a subgroup $P$ 
	with the following properties:
	\begin{enumerate}
		\item $|x|=q$ and $P$ is a $p$-group, with $q$ and $p$ distinct primes.
		\item $x\in N_G(P)\setminus C_G(P).$
		\item $C_{Z(P)}(x)\neq 1$.
	\end{enumerate}
	Then $\mathcal{E}(G)$ contains a claw in which the central vertex has order $p$ 
	and the leaves have order $q.$
\end{lemma}
\begin{proof}
	Let $y \in Z(P)$ of order $p$ such that $[x, y] = 1$.
	The group $P\langle x \rangle$ has more than one Sylow $q$-subgroup; in fact 
	it has at least three,
	$\langle x \rangle, \langle x^{y_1} \rangle, \langle x^{y_2} \rangle$
	for suitable $y_1, y_2 \in P$.
	Since $y_1, y_2 \in C_P(y)$, each of $x^{y_1}$ and $x^{y_2}$ 
	commutes with $y$.
	As $x, x^{y_1}, x^{y_2}$ generate pairwise 
	distinct Sylow $q$-subgroups of $P\langle x \rangle$, no two of 
	$x, x^{y_1}, x^{y_2}$ commute; hence no two of them generate a cyclic 
	subgroup, i.e., they are pairwise non-adjacent in $\mathcal{E}(G)$. 
	On the other hand, each of $x, x^{y_1}, x^{y_2}$ has order $q$ coprime 
	to $|y| = p$ and commutes with $y$, so $\langle x^{y_i}, y \rangle$ 
	is cyclic of order $pq$; thus each is adjacent to $y$.
\end{proof}

\begin{proposition}\label{solv-2-Frb-G-trival-dom}
	Let $G$ be a finite solvable group such that $\mathcal{E}^{**}(G)$ 
	is claw-free, $\cyc(G) = 1$, and the Fitting subgroup $F(G)$ 
	is not cyclic. Then the following hold:
	\begin{enumerate}
		\item $F(G)$ is a non-cyclic $p$-group.
		\item If $G$ is not nilpotent and $F_2(G)$ denotes the preimage 
		in $G$ of $F(G/F(G))$, then $F_2(G)$ is a Frobenius group with 
		Frobenius kernel $F(G)$, and $F_2(G)/F(G)$ is cyclic.
		\item If $G \neq F_2(G)$, then $G/F_2(G)$ is a $p$-group and 
		$G$ is a $2$-Frobenius group.
	\end{enumerate}
\end{proposition}

\begin{proof}
	Set $F = F(G)$, $F_2 = F_2(G)$. We assume that $F$ is not cyclic.
	It follows immediately from Lemma~\ref{excludedsbgs} that $F$ is a $p$-group
	which is neither cyclic nor generalized quaternion.
	
	\medskip
	\noindent{\bf{Step 1}}: For every $H \le G$ with $\gcd(|H|, p) = 1$, 
	$Z(F)H$ is a Frobenius group with kernel $Z(F)$.
	
Suppose not. Then there exist elements $x \in H$ of prime order $q$ 
with $C_{Z(F)}(x)\neq 1.$ Since $F$ is neither cyclic nor generalized 
quaternion, it contains a subgroup isomorphic to $C_p \times C_p$. 
If $[x,F]=1$, then $C_p \times C_p \times C_q \leq F\langle x\rangle 
\leq G$, contradicting Lemma~\ref{excludedsbgs}. If $[x,F]\neq 1$, 
then by Lemma~\ref{frobcent}, four non-trivial elements of 
$F\langle x\rangle$ would induce a claw in $\mathcal{E}(G)$.
Since $\cyc(G) = 1$, none of these elements is a dominating 
vertex, so this configuration survives in $\mathcal{E}^{**}(G)$ and
induces a claw in $\mathcal{E}^{**}(G)$.

	\medskip
	\noindent{\bf{Step 2}}: For every prime $q \neq p$, a Sylow $q$-subgroup 
	of $G$ is cyclic.
	
	Let $Q$ be a Sylow $q$-subgroup of $G$, $q \neq p$. By Step~1, $Z(F)Q$ is 
	a Frobenius group with kernel $Z(F)$ and complement $Q$. Hence $Q$ is cyclic or generalized quaternion  \cite[V 8.7 Hauptsatz]{Endliche-gruppen-I-Huppert}. The 
	second possibility is excluded by Lemma~\ref{excludedsbgs}.
	
	\medskip
	\noindent{\bf{Step 3}}: $F_2/F$ is cyclic, and $G/F_2$ is abelian.
	
	Since $F_2/F = F(G/F)$ is nilpotent, it is the direct product of its 
	Sylow subgroups. The $p$-part of $F_2/F$ is trivial, since $F$ already 
	contains all normal $p$-subgroups of $G$. For each prime $q \neq p$ 
	dividing $|F_2/F|$, the corresponding Sylow $q$-subgroup of $F_2/F$ 
	corresponds under the natural projection to a Sylow $q$-subgroup of $G$, 
	which is cyclic by Step~2. Hence $F_2/F$ is a direct product of cyclic 
	groups of pairwise coprime orders, and therefore cyclic.
	
	Since $F_2/F$ is cyclic, $\operatorname{Aut}(F_2/F)$ is abelian. Moreover, 
	$G/F_2$ acts on $F_2/F$ by conjugation with trivial kernel, since 
	$F_2/F = F(G/F)$ is self-centralizing in the solvable group $G/F$ 
\cite[III 4.2 Satz]{Endliche-gruppen-I-Huppert}. This gives an embedding
$
	G/F_2 \;\hookrightarrow\; \operatorname{Aut}(F_2/F),
$
	hence $G/F_2$ is abelian.
	
	\medskip
	\noindent{\bf{Step 4}}: If $F_2 \neq G$, then $G/F_2$ is a 
	$p$-group and $G$ is a $2$-Frobenius group.
	
	We may decompose $G/F$ as a semidirect product $(N/F)(X/F)$, 
	where $N/F$ is a Hall $p'$-subgroup of $G/F$ and $X/F$ is a Sylow 
	$p$-subgroup of $G/F$.
	
	We claim that $X/F$ acts fixed-point-freely on $N/F$. Suppose not. 
	Then $G$ contains two commuting elements $x$ and $y$ of orders $q$ and $p$ respectively,  with
	$q \neq p$. Since, by Step~1, $C_{Z(F)}(x)=1,$ it follows that $y\notin Z(F).$ Let
	$Y=\langle y\rangle Z(F)$ and $\tilde Y=C_Y(y)$. Since $Z(F)\neq 1,$ $y\notin Z(F)$ has order $p$ and $\tilde Y$ is not a generalized quaternion group, we have $C_p\times C_p\leq \tilde Y$.
	In particular,
	$[x,\tilde Y]\neq 1$ (otherwise $C_p\times C_p\times C_q\leq G$). It follows from Lemma \ref{frobcent} that, since  $C_{\tilde Y}(x)\neq 1,$  $\mathcal E^{**}(G)$ contains a claw, a contradiction. 

	Hence $G/F$ is a Frobenius group with kernel $N/F$ and complement $X/F$. 
	By Thompson's theorem \cite[V 8.7 Hauptsatz]{Endliche-gruppen-I-Huppert}, the kernel of a Frobenius group is nilpotent, so 
	$N/F$ is a nilpotent normal subgroup of $G/F$. In particular $N \leq F_2$, 
	and since $N/F$ is the full Hall $p'$-subgroup of $G/F$ and $F_2/F$ 
	contains no $p$-part, we conclude $N = F_2$.
	Therefore $G/F_2 \cong X/F$ is a $p$-group. Finally, $G$ is a 
	$2$-Frobenius group: $F$ is the Frobenius kernel of $F_2$, and $F_2/F$ is 
	the Frobenius kernel of $G/F$.
	This completes the proof.
\end{proof}

\begin{proposition}\label{prop:FH-p-group}
	Let $G$ be a finite group and assume that $\mathcal{E}^{**}(G)$ is claw-free. 
	If $F(G)$ is not cyclic, then $F(G/\cyc(G))=F(G)/\cyc(G)$ is not cyclic.
\end{proposition}

\begin{proof}
	Let $F=F(G)$ and $C=\cyc(G).$ Since $C\leq Z(G),$ we have $F(G/C)=F/C.$
	Suppose that $F$ is non-cyclic. Then, for a suitable prime $p$, a Sylow $p$-subgroup $P$ of $F$ is non-cyclic.
	If $P\cap C=1$, then $PC/C \cong P$ is a non-cyclic subgroup of $F/C.$
	So assume $P\cap C\neq 1.$ By Lemma~\ref{excludedsbgs}, $P$ is a generalized quaternion group, and $P\cap C=Z(P)$ is cyclic of order $2$. In particular, $P/(P\cap C)$ is not cyclic, and hence $PC/C$ is a non-cyclic subgroup of $F/C.$
	In both cases, $F/C$ is non-cyclic.
\end{proof}

\begin{proof}[Proof of Theorem~\ref{solvable}]
	By Corollary~\ref{behacyc}, $\mathcal{E}^{**}(G/\cyc(G))$ is claw-free and has no non-trivial dominating elements. If $F(G)$ is cyclic, then by Theorem \ref{Thm:F(G)-cyclic}, $G/F(G)$ is also cyclic. Moreover, by Proposition~\ref{prop:FH-p-group}, if $F(G)$ is non-cyclic, then $F(G/\cyc(G))$ is also non-cyclic. So the conclusion follows from Proposition~\ref{solv-2-Frb-G-trival-dom}.
\end{proof}

	\begin{proof}[Proof of Corollary~\ref{nilpotent}]
		Assume that $G$ is a non-cyclic nilpotent group and that $\mathcal{E}^{**}(G)$ is claw-free. Set $C=\cyc(G).$ By Theorem~\ref{solvable}, $F(G)/C = F(G/C)$ is a non-cyclic $p$-group. Since $G$ is nilpotent, $F(G)=G$, and since $C$ is cyclic, it follows that $G=P \times Z,$ where $P$ is a non-cyclic $p$-group and $Z$ is a cyclic $p$-complement in $C$.
		
		To conclude, it suffices to prove that if $G=P\times Z,$ where $P$ is a non-cyclic $p$-group and $Z$ is a cyclic group with $\gcd(p,|Z|)=1,$ then $\mathcal{E}^{**}(G)$ is claw-free if and only if $\mathcal{E}^{**}(P)$ is claw-free. For this purpose, notice that $\cyc(P\times Z)=\cyc(P)\times Z.$ Hence $G/C\cong P/\cyc(P)$ and, by Corollary~\ref{behacyc}, the following are equivalent: $\mathcal{E}^{**}(G)$ is claw-free; $\mathcal{E}^{**}(G/C)$ is claw-free; $\mathcal{E}^{**}(P/\cyc(P))$ is claw-free; $\mathcal{E}^{**}(P)$ is claw-free.
	\end{proof}

\begin{proof}[Proof of Theorem~\ref{pgroups}]
	First assume that $G$ is a generalized quaternion group. Then $\cyc(G)=Z(G)$ and $G/Z(G)$ is a dihedral group. It can be easily seen that the proper enhanced power graph of a dihedral group is claw-free, so $\mathcal{E}^{**}(G)$ is claw-free. Hence we may assume that $G$ is neither cyclic nor generalized quaternion. But then
	$\mathcal{E}^{**}(G)$ is claw-free if and only if the reduced power graph $\mathcal P^*(G)$ is claw-free, and  \cite[Proposition 2.5]{claw-free-pwr-graph} implies that $G$ is one of the groups listed in (i)--(v).
\end{proof}

\section{Non solvable groups}



\begin{lemma}\label{Thm:non-solvable-claw-free-Z(G)=e}
	Let $G$ be a finite non-solvable group and suppose that the identity is the unique dominating vertex of $\mathcal{E}(G)$. If $\mathcal{E}^{**}(G)$ is claw-free, then $Z(G)=1.$
\end{lemma}

\begin{proof}
	Assume, by contradiction, that $Z(G)\neq 1$ and let $p$ be a prime divisor of $|Z(G)|.$ If $q\neq p$, and $Q$ is a Sylow $q$-subgroup of $G$, then, by Lemma~\ref{excludedsbgs}, $Q$ cannot be generalized quaternion. If $Q$ is not cyclic, then $C_q \times C_q\leq Q$, and consequently $C_p \times C_q \times C_q\leq Z(G)Q,$ in contradiction with Lemma~\ref{excludedsbgs}.
	So $Q$ is cyclic. Assume that $Q$ is not normal. Then there exist at least three distinct Sylow $q$-subgroups $\langle y_1\rangle, \langle y_2\rangle, \langle y_3\rangle$ in $G$. But then, if $x$ is an element of order $p$ in $Z(G)$, the elements $x, y_1, y_2, y_3$ induce a claw with center $x$ and leaves $y_1, y_2, y_3$ in $\mathcal{E}^{**}(G)$, a contradiction.
	We have thus proved that for every $q\neq p,$ the Sylow $q$-subgroup of $G$ is normal. The product of these Sylow subgroups is a nilpotent normal subgroup of $G$ and has index a power of $p$ in $G$, so $G$ is solvable, contradicting our assumption.
\end{proof}

 \begin{lemma}\label{nocenters}
 	Let $G$ be a finite group. Suppose that the identity element is the unique dominating vertex of $\mathcal{E}(G)$ and that $\mathcal{E}^{**}(G)$ is claw-free. Then $Z(H)=1$
 	for every non-solvable subgroup $H$ of $G$.
 \end{lemma}
 
 \begin{proof}
 Let $H$ be a non-solvable subgroup of $G$. If $\cyc(H)=1$, then the conclusion follows from Lemma~\ref{Thm:non-solvable-claw-free-Z(G)=e}. Otherwise, let $y$ be an element of prime order $q$ in $\cyc(H)$. Since $H$ is non-solvable, it contains a non-cyclic Sylow $p$-subgroup $P$ for some prime $p$ \cite[IV 2.11 Satz]{Endliche-gruppen-I-Huppert}. By Lemma~\ref{excludedsbgs}, $P$ is not generalized quaternion either. Hence $P$ contains three elements $x_1, x_2, x_3$ of order $p$ with $\langle x_i \rangle \neq \langle x_j \rangle$ for $i \neq j$. Since $y\in\cyc(H)$, the Sylow $q$-subgroup of $H$ is cyclic, so $q\neq p$. Therefore $y$ commutes with each $x_i$, and $(y, x_1, x_2, x_3)$ is a claw in $\mathcal{E}^{**}(G)$, a contradiction.
 \end{proof}

\begin{proposition}\label{prop:Fstar-F}
	Let $G$ be a finite non-solvable group. Suppose that the identity is the unique dominating vertex of $\mathcal{E}(G)$ and that  $\mathcal{E}^{**}(G)$ is claw-free. Then 
$G$ is an almost simple group.
\end{proposition}

\begin{proof}
	We prove the statement by induction on the order of $G$. If $H$ is a proper non-solvable subgroup of $G$, then  by Lemma~\ref{nocenters}, $\cyc(H)\leq Z(H)=1,$ and therefore, by induction, $H$ is almost simple. Hence every proper subgroup of $G$ is either solvable or almost simple.

Recall that $F^*(G)$ is the central product 	$F(G)E(G)$ of the Fitting subgroup and the subgroup $E(G)$  generated by
all components of $G.$ Suppose that $F^*(G)\neq F(G).$ It follows from Lemma \ref{nocenters} that $Z(F^*(G))=1,$ hence $E(G)=S_1\times \dots \times S_r$ and $F^*(G)=F(G)\times E(G).$ Let $S=S_1.$ We claim that $F^*(G)=S.$ Otherwise let $1\neq g \in S_2\times \dots S_r\times F(G).$
Then $\langle S, g\rangle$ is a non-solvable subgroup of $G$ with non-trivial center, in contradiction with Lemma \ref{nocenters}.
Since $C_G(F^*(G))\leq F^*(G)$ and $F^*(G)=S$ is simple with trivial center, we have $C_G(S)=1$. Thus
$S\leq G\leq \operatorname{Aut}(S),$
and therefore $G$ is almost simple.

So we may assume that  $F^*(G)=F(G).$
By Lemma \ref{excludedsbgs}, $F(G)$ contains no subgroup isomorphic to $C_p\times C_p \times C_q,$ so either $F(G)$ is cyclic or it is a $p$-group. However in the first case, $G/C_G(F(G))=G/F(G) \leq \aut(F(G)),$ which implies that $G/F(G)$ is abelian and, consequently, $G$ is solvable.  Hence $F(G)$ is a noncyclic $p$-group. 

Let $R(G)$ denote the solvable radical of $G$, and let $Q$ be a Sylow $q$-subgroup of $R(G)$ for some prime $q$. The Frattini argument yields
$G=R(G)N_G(Q).$ Since $Q$ is nilpotent normal subgroup of $N_G(Q),$ $N_G(Q)$ is not almost simple.  If $N_G(Q)\neq G$, then $N_G(Q)$ is solvable. It follows that
$G=R(G)N_G(Q)$ is also solvable, a contradiction. Consequently,
$N_G(Q)=G.$
Thus every Sylow subgroup of $R(G)$ is normal in $G$. Hence $R(G)$ is nilpotent, and consequently $R(G)\leq F(G).$ Since $F(G)\leq R(G),$ we conclude $F(G)=R(G).$ Hence a minimal normal subgroup  $N/R(G)$ of $G/F(G)$ is non-abelian. Thus $N$ is neither solvable nor almost simple, and thus implies that $N=G.$ We have proved that $G/F(G)$ is a non-abelian simple group. 

Since $F(G)$ is a non-cyclic $p$-group, $Z=Z(F(G))$ is a non-trivial $p$-group. We claim that $Z$ is non-cyclic. Indeed, suppose that $Z$ is cyclic, and let $K=C_G(Z)$. Since $G/K\leq\operatorname{Aut}(Z)$ is abelian, $G/K$ is solvable. Since $G$ is non-solvable, $K$ is also non-solvable. But $Z\leq Z(K)$, so $K$ has non-trivial center, contradicting Lemma~\ref{Thm:non-solvable-claw-free-Z(G)=e}. Hence $Z$ is non-cyclic and $C_p\times C_p\leq Z.$ 
Let $T$ be a Sylow $q$-subgroup of $G$, where $q\neq p$. Since $[T,Z]\neq 1$ (otherwise $C_p\times C_p\times C_q)$, we deduce from Lemma \ref{frobcent} that $ZT$ is a Frobenius group with kernel $Z$ and complement $T$.
Therefore, $T$ is either cyclic or generalized quaternion.
The second possibility is excluded by Lemma \ref{excludedsbgs}. 
In particular, every Sylow $q$-subgroup of $G/F(G)$ is cyclic for $q\neq p$. On the other hand, the Sylow $2$-subgroups of a non-abelian simple group are never cyclic. Therefore, we must have $p=2$. 

Let $P$ be a Sylow $2$-subgroup of $G$. Notice that the reduced power graph of $P$ must be claw-free since a claw in this graph would be also a claw in $\mathcal E^{**}(G).$ In particular, by \cite[Lemma~2.3]{claw-free-pwr-graph}, no subgroup of $P$ is isomorphic to $C_4\times  C_2\times  C_2.$

Observe that $P$ is non-abelian. Indeed, if $P$ were abelian, then
$P \leq C_G(F(G)) = F(G),$
which contradicts the fact that $P/F(G)$ is isomorphic to a Sylow $2$-subgroup of the non-abelian simple group $G/F(G)$.
Furthermore, $P$ cannot be dihedral. Indeed, if $P$ were dihedral, then $F(G)$ would be a normal subgroup of $P$ of index at least $4$ and therefore $F(G)$ would be cyclic. Since $\mathcal{P}^*(G)$ is claw-free, it follows from \cite[Proposition~2.6]{claw-free-pwr-graph} that the exponent of $P$ is at most 4.

We next show that \(F(G)\) is elementary abelian. 
Suppose, to the contrary, that $F(G)$ contains an element $g$ of order 4. Then the subgroup $X$ generated by the elements of order 2 in $Z(F(G))$ has order at most 4 (otherwise $C_P(g)$ contains a subgroup isomorphic to $C_4\times C_2\times C_2)$. Since, by Lemma \ref{Thm:non-solvable-claw-free-Z(G)=e}, $Z(G)=1,$ we have that $C_G(X)\neq G$. It follows that $C_G(X)$ is solvable. Since $G/C_G(X)\leq \aut(X)$ is also solvable, we deduce that $G$ itself is solvable, a contradiction.

Now we may conclude the proof by using the same argument as in the proof of Theorem 4.1 in \cite{claw-free-pwr-graph}. By  \cite[Proposition 2.8]{claw-free-pwr-graph} $P=N\langle x\rangle,$ with $N$ an abelian normal subgroup of $P.$ Moreover, since $P/F(G)$ is isomorphic to a Sylow 2-subgroup of a non-abelian simple group, $P/F(G)$ is not cyclic.

First assume that $N\leq C_4\times C_4.$ This implies $F(G)\leq C_2\times C_2 \times C_2.$ Since $G/F(G)$ is a non-abelian simple group, the only possibility is that $F(G)\cong C_2\times C_2\times C_2$ and $G/F(G)\cong {\rm{SL}}(3,2).$ In particular, $G$ is perfect of order 1344, with no element of order 8. There is a unique possibility: $G\cong \rm{{AGL}}(3,2)$, that can be exclude since $G$ would contains 3 elements $x_1,x_2,x_3$ of order 6, which pairwise do not generate a cyclic subgroup, and with $x_1^3=x_2^3=x_3^3.$

Since $N$ cannot be cyclic (otherwise $|F(G)|\leq 4$ and $G$ would be solvable), by \cite[Corollary 2.4]{claw-free-pwr-graph}, $N$ is an elementary abelian 2-group. 
Since $P/F(G)$ is not cyclic, $N$ is not contained in $F(G).$ This implies  $N\not\leq C_P(F(G)).$ 
Let $y\in F(G)\setminus C_P(N).$ Then $\langle y\rangle N$ is non-abelian and therefore there exists $n\in N$ such that $z=yn$ has order 4. Since $P$ contains no subgroup isomorphic to $C_4\times C_2\times C_2,$ we deduce $|C_N(z)|\leq 4.$
Since $F(G)$ is abelian, $F(G)\cap N\leq C_N(z)=C_N(y)\leq C_2\times C_2$. This implies $F(G)\leq C_2\times C_2 \times C_2$. We can conclude as in the previous paragraph.
\end{proof}

\begin{theorem}\label{thm:almost-simple}
	Let $G$ be a finite non-solvable group. If the proper enhanced power graph $\mathcal{E}^{**}(G)$ is claw-free, then the quotient group $G/\cyc(G)$ is almost simple.
\end{theorem}

\begin{proof}
	By Corollary  \ref{behacyc}, we may assume that $\mathcal E^{**}(G/\cyc(G))$ is claw-free and $\cyc(G/\cyc(G))=1.$ So the conclusion follows from the previous proposition.
\end{proof}

\section{Proof of Theorem \ref{nonsolvable}}

The following theorem plays a key role in the study of finite simple groups whose proper enhanced power graph is claw-free.

\begin{theorem}\label{thm:classification}
	Let $G$ be a finite non-abelian simple group.  Assume that $G$ contains neither a subgroup isomorphic to $Q_8$ nor a subgroup
	isomorphic to $C_2\times C_2\times C_p$, for any odd prime $p$.
	Then $G$ is isomorphic to one of the following groups:
	\[
	\PSL(2,q),\qquad
	\Sz(2^{2n+1})\ (n\geq 1),\qquad
	J_1,
	\]
	where $q\geq4$ is a prime power.
\end{theorem}
\begin{proof}
	Assume first that no subgroup of $G$ is isomorphic to $C_2\times C_2\times C_p$, with $p$ an odd prime. Then $C_G(V)$ is a 2-group for every subgroup $V$ of $G,$ with $V\cong C_2\times C_2.$
By	\cite[Theorem 1]{Syskin1978},
$G$ is isomorphic to one of the following groups:
	\begin{enumerate}[(1)]
		\item $\operatorname{PSL}(2,q)$, $\operatorname{Sz}(q)$, $J_{1}$, $M_{11}$, $\operatorname{PSL}(3,4)$, or ${}^{2}F_{4}(2)'$;
		\item $\operatorname{PSL}(3,q)$, $\operatorname{PSU}(3,q)$, or $G_{2}(q)$, where $q \geq 3$ is an odd prime power.
	\end{enumerate}
	 We now eliminate from that list the groups
	which contain $Q_8$.
	
	The group $M_{11}$ has a maximal subgroup of shape
	$2.S_4\cong\GL(2,3)$; its subgroup $\SL(2,3)$ contains a normal quaternion
	subgroup of order eight.  Thus $Q_8\leq M_{11}$.  
	The group $\PSL(3,4)$ also contains $Q_8$.  
	The Tits group ${}^2F_4(2)'$ contains a subgroup isomorphic to
	$\PSL(3,3)$; the latter contains the block-diagonal subgroup
	$\SL(2,3)$ and hence contains $Q_8$.
	
	If $q$ is odd, both $\PSL(3,q)$ and $\PSU(3,q)$ contain a natural rank-one
	subgroup isomorphic to $\SL(2,q)$.  A Sylow $2$-subgroup of $\SL(2,q)$ is
	generalized quaternion and therefore contains $Q_8$.  Finally, $G_2(q)$ in
	odd characteristic contains a subsystem subgroup of type $A_2$, and hence a
	subgroup obtained from $\SL(3,q)$ by factoring out a central subgroup of odd
	order; this again contains $Q_8$.  
	Consequently only $\PSL(2,q)$, $\Sz(2^{2n+1})$, and $J_1$ remain.
\end{proof}

\begin{corollary}\label{semplici}Let $S$ be a non-abelian simple group. If $\mathcal E^{**}(S)$ is claw-free, then $S\cong \PSL(2,q)$, with $q\geq 4.$
	\end{corollary}
\begin{proof}
Let $S$ be a finite non-abelian simple group and suppose that $\mathcal E^{**}(S)$ is claw-free. By Lemma~\ref{excludedsbgs} and Theorem \ref{thm:classification}, 	$S$ is isomorphic to one of the following groups:
$
\PSL(2,q)$ with $q\geq 4,$
$\Sz(2^{2n+1})$,
$J_1.$ By \cite[Lemma 5.3]{claw-free-pwr-graph}, $\Sz(2^{2n+1})$ contains a claw with center of order 2, and leaves of order 4, so $\mathcal{E}^{**}(\Sz(2^{2n+1}))$ is not claw-free.
The group $J_1$ contain a subgroup $H\cong A_5\times C_2$, so $\mathcal E^{**}(J_1)$ is not claw-free by Lemma \ref{nocenters}.
\end{proof}

\begin{theorem}
	\label{PSL}
	Let
$S=\PSL(2,q)$ with $q=p^n\geq 5,$ where $n$ is a prime, and let $G$ be an almost simple group
	with socle $S$. 
	Then $\mathcal{E}^{**}(G)$ is claw-free if and only if 
		$S\leq G\leq \operatorname{PGL}(2,q).$ 
	\end{theorem}

\begin{proof}
	We divide the proof into several steps.
	
	\medskip
	
	\textbf{Step 1}: Groups contained in $\operatorname{PGL}(2,q)$.
	
	Assume that $S\leq G\leq \operatorname{PGL}(2,q)$. By\cite[Sätze 8.2, 8.3 and 8.5]{Endliche-gruppen-I-Huppert}, $G$ admits the group partition $\mathcal P=\{A^x,B^x,P^x\mid x\in G\},$ where $P$ is elementary abelian of order $q$, while $A$ and $B$ are cyclic of orders $$\frac{q-1}{d}\qquad\text{and}\qquad \frac{q+1}{d},$$ respectively, with $$d=\begin{cases}1,&\text{if }G=\operatorname{PGL}(2,q),\\\gcd(q-1,2),&\text{if }G=S.\end{cases}$$ Here “group partition” means that distinct members of $\mathcal P$ intersect precisely in the identity and that their union is $G$.
	
	We claim that every connected component of $\mathcal E^{**}(G)$ is complete. Indeed, let $u,v\neq 1$ be adjacent. Then $\langle u,v\rangle$ is cyclic. If $z$ is a generator of $\langle u,v\rangle$, the element $z$ belongs to a unique member $X\in\mathcal P$. Since $X$ is a subgroup and $u,v\in\langle z\rangle$, we have $u,v\in X$.
	If $X$ is a conjugate of $A$ or $B$, then $X$ is cyclic, so $X\setminus{1}$ is a clique. If $X$ is a conjugate of $P$, then $X$ has exponent $p$, and $u$ and $v$ are adjacent precisely when they belong to the same subgroup of order $p$. Consequently, $X\setminus{1}$ is the disjoint union of the cliques $\{L\setminus{1}\mid L\leq X, |L|=p\}.$
	
	It follows that $\mathcal E^{**}(G)$ is a disjoint union of complete graphs. In particular, it is claw-free.

	\medskip
	\noindent
	\textbf{Step 2}: The outer automorphism structure.
	
	We now prove the converse.  Recall that
$
	\operatorname{Aut}(S)
	=\operatorname{P\Gamma L}(2,q)
	=\operatorname{PGL}(2,q)\rtimes\langle\varphi\rangle,
$
	where
$
	\varphi:a\longmapsto a^\ell
$
	is the standard field automorphism of order $n$.  Moreover,
$
	\operatorname{Out}(S)
	\cong C_{(2,q-1)}\times C_n.
$
	When $q$ is odd, denote by $\delta$ the non-trivial diagonal
	outer automorphism.
	Assume that
$
	G\not\leq \operatorname{PGL}(2,q).
$
	Thus the image of $G/S$ in the field-automorphism factor is
	non-trivial.  We show that
	the graph $\mathcal{E}^{**}(G)$ contains an induced claw.
	
	\medskip
	\noindent
	\textbf{Step 3}: A field automorphism of odd prime order.
	
	Suppose that the field part of $G/S$ has order divisible by an
	odd prime $r$.  Since the diagonal outer factor has order at
	most $2$, the group $G/S$ contains a pure field automorphism of
	order $r$.  Thus $G$ contains
$
	\alpha=\varphi^{n/r},
$
	after replacing $\alpha$ by a suitable power if necessary.
	Put
$
	q_0=p^{n/r}.
$
	The automorphism $\alpha$ centralizes the natural subfield
	subgroup $
	S_0\cong\operatorname{PSL}(2,q_0).$
	The group $S_0$ contains at least three distinct involutions.
	Indeed, this is immediate when
$
	\operatorname{PSL}(2,2)\cong S_3$
or $	\operatorname{PSL}(2,3)\cong A_4.
$
	If $q_0\geq4$, then $\operatorname{PSL}(2,q_0)$ is non-abelian
	simple, and the conjugacy class of an involution has size at
	least $3$.
	Choose three distinct involutions
$
	t_1,t_2,t_3\in S_0.
$
	Since $\alpha$ centralizes $S_0$ and $r$ is odd,
each $t_i$ is adjacent to $\alpha$.
	On the other hand, two distinct involutions cannot generate a
	cyclic group. It
	follows that
$
	\{\alpha,t_1,t_2,t_3\}
$
	induces a claw with center $\alpha$.
	
	\medskip
	
	Consequently, if $\mathcal{E}^{**}(G)$ is claw-free, the field
	part of $G/S$ must be a $2$-group.
	
	\medskip
	\noindent
	\textbf{Step 4}: A pure field involution.
	
	Suppose that $G$ contains the pure field involution
$
	\tau=\varphi^{n/2}.
$
	Set
$
	q_0=p^{n/2}.
$
	Then $\tau$ centralizes the natural subfield subgroup
$
	S_0\cong\operatorname{PSL}(2,q_0).
$
	
Since we are assuming $q\geq 5,$ we have that $q_0\geq3$.  We claim that $S_0$ contains
	three elements
$
	y_1,y_2,y_3
$
	of the same odd prime order $r$, lying in three distinct
	subgroups of order $r$.
	For $q_0=3$, this follows from
$
	\operatorname{PSL}(2,3)\cong A_4,
$
	which has four subgroups of order $3$.  If $q_0\geq4$, then
	$S_0$ is non-abelian simple.  Choose an odd prime divisor $r$
	of $|S_0|$ and a subgroup $R$ of order $r$.  The subgroup $R$
	is not normal in $S_0$, so its conjugacy orbit has size at
	least $3$.  Therefore $S_0$ contains at least three distinct
	conjugates of $R$. Choose $1\neq y_i\in R_i,$
	where $R_1,R_2,R_3$ are distinct subgroups of order $r$.
	Since $\tau$ centralizes $S_0$ and $r$ is odd,
	$\{\tau,y_1,y_2,y_3\}$
	induces a claw.
	
%
	
	\medskip
	\noindent
	\textbf{Step 5}: A diagonal-field involution.
	
	It remains to consider the case in which $q$ is odd and $G/S$
	contains the diagonal-field involution
$
	\overline{\delta}\,
	\overline{\varphi}^{\,n/2},
$
	but not the pure field involution
$
	\overline{\varphi}^{\,n/2}.$
The Sylow $2$-subgroup of the corresponding subgroup of $\operatorname{Aut}(S)$ is
 semidihedral \cite[Lemma 2.3]{goren}.
	Since $q$ is an odd square, these Sylow subgroups have order at
	least $16$. Since every semidihedral group of order at least
	$16$ contains a quaternion subgroup of order $8$, we conclude that $\mathcal E^{**}(G)$ is not claw-free.

\medskip

	The structure
	$
	\operatorname{Out}(S)
	\cong C_2\times C_n
$
	shows that the preceding cases exhaust all groups which induce
	a non-trivial field automorphism.  Therefore the only
	claw-free groups are those listed in the statement.
\end{proof}

\begin{proof}[Proof of Theorem \ref{nonsolvable}]
	
First assume that $G=X \times Y$, with $X$ and $Y$ satisfying the hypotheses of the statement. Then $G/\cyc(G)$ is isomorphic to $\PSL(2,q)$ or $\PGL(2,q)$, with $q\geq 5$. By Theorem~\ref{PSL}, $\mathcal{E}^{**}(G/\cyc(G))$ is claw-free. By Corollary~\ref{behacyc}, $\mathcal{E}^{**}(G)$ is also claw-free.

Conversely,	suppose that
	$G$ is a finite non-solvable group and that $\mathcal E^{**}(G)$ is claw-free. Let $C=\cyc(G).$ By Theorem \ref{thm:almost-simple}, $H\cong G/C$ is an almost simple group, so in particular $C=Z(G)$, and, by Corollary \ref{semplici}, $\soc(G/C)=\PSL(2,q).$ Since $\PSL(2,4)\cong \PSL(2,5)$, we may assume without loss of generality
	that $q\geq 5,$ so, by Theorem \ref{PSL}, $H$ is isomorphic to $\PSL(2,q)$ or $\PGL(2,q)$ with $q\geq 5.$ Moreover if a prime $p$ divides $|C|$, then a Sylow $p$-subgroup of $G$ is either cyclic or generalized quaternion. Since a Sylow 2-subgroup of $H$ is neither cyclic or generalized quaternion, we deduce that if $2$ divides $C$, then a Sylow 2-subgroup of $G$ is generalized quaternion and the 2-part of $C$ has order 2.

	We now consider the two possible quotients separately.
	
	\medskip
	\noindent
	\textbf{Case 1: $H\cong\operatorname{PSL}(2,q)$.}

Let
$
	D=G'.
$
Since $H=G/C$ is perfect,
	$G=DC.
$
	Because $C$ is central,
$
	D=G'=(DC)'=D'$,
	so $D$ is perfect.  If
$
	Z_0=D\cap C,
$
	then
$
Z_0\leq Z(D)\cap D'
$
and
$
D/Z_0\cong H.
$
	Thus $D$ is a perfect stem central extension of $H$.
	
	The possible groups $D$ are controlled by the Schur multiplier
	of $H$.  The standard multiplier formula gives
	$
	M(\operatorname{PSL}(2,q))
	\cong C_{(2,q-1)}
	$
	apart from the exceptional cases
	$
	M(\operatorname{PSL}(2,4))\cong C_2,$
$M(\operatorname{PSL}(2,9))\cong C_6.
$ (see, for example, 
	\cite[Kapitel V, Satz 25.7]{Endliche-gruppen-I-Huppert}).
	
	Suppose first that $q$ is odd.  Apart from $q=9$, it follows
	that
$
	Z_0=1
$ or $
	Z_0\cong C_2,
$
	and consequently
$
	D\cong\operatorname{PSL}(2,q)$ or
$	D\cong\operatorname{SL}(2,q).
$
	
	When $q=9$, one must also exclude the $3$-part of the exceptional
	Schur multiplier.  Indeed,
$
	\operatorname{PSL}(2,9)\cong A_6
$
	has a Sylow $3$-subgroup isomorphic to
$
	C_3\times C_3, 
$ and therefore a Sylow 3-subgroup of $G$ cannot be cyclic.

	Finally, if $q\geq8$ is even, the Schur multiplier of
	$\operatorname{PSL}(2,q)$ is trivial.  Hence
$
	Z_0=1$
and 
$D\cong\operatorname{PSL}(2,q).
$

We have seen that the subgroup $Z_0$ is either trivial or has order $2$.
	Moreover, when $|Z_0|=2$, the $2$-part of $C$ is precisely
	$Z_0$.  Since $C$ is cyclic, it follows that $Z_0$ has a
	complement $X$ in $C$.
	Now
$
	G=DC=DX, 
	D\cap X=1,
$ and
	$[D,X]=1.
$
	Therefore
	$G\cong D\times X.
$

If $|D|$ and $|X|$ had a common prime divisor $r$, the Sylow $r$-subgroup of $G$ could be neither cyclic nor generalizeq quaternion. 	Hence $
\gcd(|D|,|X|)=1.
$

	Taking $Y=D$ proves all the assertions when the quotient
	is $\operatorname{PSL}(2,q)$.
	
	\medskip
	\noindent
	\textbf{Case 2: $H\cong\operatorname{PGL}_2(q)$.}
	
	If $q$ is even, then $
	\operatorname{PGL}(2,q)
	=\operatorname{PSL}(2,q)
$.  We may therefore
	assume that $q$ is odd.  Put
$H=\operatorname{PGL}(2,q),$
$S=H'=\operatorname{PSL}_2(q).
$
	Then
$
	H/S \cong C_2.
$
	
	Let $X$ be the Hall $2'$-subgroup of $C$.  We claim that $X$
	splits off as a direct factor of $G$.
	
	For completeness, recall the universal coefficient sequence
	for central extensions with trivial action:
	\[
	0\longrightarrow
	\operatorname{Ext}^{1}(H_{\mathrm{ab}},X)
	\longrightarrow
	H^2(H,X)
	\longrightarrow
	\operatorname{Hom}(M(H),X)
	\longrightarrow0.
	\tag{1}
	\]
	Here
$H_{\mathrm{ab}}\cong C_2,
$
	and the Schur multiplier of $\operatorname{PGL}(2,q)$ is a
	$2$-group.  Since $X$ has odd order, both the left-hand and the
	right-hand terms in (1) vanish.  Hence the odd part of the
	central extension splits.  Since $X$ is central, the splitting
	is direct.
	
Hence
$
	G\cong X\times Y,
$
	where $X$ is the odd-order part of $C$, the group $X$ is cyclic,
	and $
	Y/B\cong\operatorname{PGL}(2,q)
$
	for some
$	B\leq Z(Y)$ with $|B|\leq 2.$

As in Case~1, if an odd prime $r$ divided both $|X|$ and
$|\operatorname{PGL}(2,q)|$, a Sylow $r$-subgroup of $G$ would not be cyclic.
Hence $\gcd(|X|,|Y|)=1.$

	If $B=1$, then
$
	Y\cong\operatorname{PGL}(2,q),
$
	and there is nothing more to prove.
	
	Suppose now that $
	B\cong C_2.$ 
	A Sylow $2$-subgroup of $Y$ cannot be cyclic, since its quotient
	contains a dihedral Sylow $2$-subgroup of
	$\operatorname{PGL}(2,q)$.  Hence a Sylow $2$-subgroup of $Y$
	must be generalized quaternion.
	
	Let $N$ be the full inverse image in $Y$ of
$
S=\operatorname{PSL}(2,q).
$
	Then
$
	N/B\cong S.
$
	The extension
	\[
	1\longrightarrow B
	\longrightarrow N
	\longrightarrow S
	\longrightarrow1
	\]
	cannot split.  Indeed, if it split, then
$
	N\cong C_2\times S
$
and a Sylow $2$-subgroup of $N$ would then contain a Klein
	four-group.  This is impossible because every subgroup of a
	cyclic or generalized quaternion $2$-group is cyclic or
	generalized quaternion, and in particular contains no subgroup
	isomorphic to $C_2\times C_2$.
	It follows that
$
	N\cong\operatorname{SL}_2(q).
$
	In particular,
$
	B=Z(N)\leq N'=N\leq Y',
$
	so the extension of $\operatorname{PGL}(2,q)$ by $B$ is a stem
	extension.
	
	For every odd $q$, there are exactly two non-isomorphic
	non-split stem central extensions of
	$\operatorname{PGL}(2,q)$ by a group of order $2$.  One has
	semidihedral Sylow $2$-subgroups, whereas the other has
	generalized quaternion Sylow $2$-subgroups (see
	\cite[Result~2.21]{GiuliettiKorchmaros}).  The latter is denoted
	here by  $\operatorname{SL}^{(2)}(2,q).$

	Our hypothesis excludes the extension with semidihedral Sylow
	$2$-subgroups.  Therefore
$
	Y\cong \operatorname{SL}^{(2)}(2,q).$
	\end{proof}

\section*{Acknowledgements} 
SB gratefully acknowledges the financial support provided by the NBHM research project (Reference No.~02011/29/2025NBHM(RP)/R\&DII/11951). He sincerely thanks the National Board for Higher Mathematics (NBHM), India, for funding this research. SB also acknowledges the excellent research environment and facilities provided by Dhirubhai Ambani University, Gandhinagar, Gujarat. 

\subsection*{Declaration of competing interest}
The authors have no relevant financial or non-financial interests to disclose.

\subsection*{Data availability}
Data sharing is not applicable to this article as no datasets were generated or analyzed 
during the present study.

\bibliographystyle{amsplain}
\bibliography{gen-inv-lcpb.bib}

\end{document}